%% file: paper.tex
\RequirePackage{silence}
\documentclass[10pt, a4paper,twocolumn]{scrartcl}

\usepackage{geometry}
\usepackage[noadjust]{cite}

\usepackage{amsmath}
\usepackage{amssymb}
\usepackage{amsthm}

\usepackage{tikz}
  \usetikzlibrary{positioning}

\usepackage{pgfplots}
  \pgfplotsset{compat = 1.13,
    colormap name = viridis,
    unbounded coords = jump}
    
\tikzset{every picture/.style={/utils/exec={\normalfont}}}
  
\usepackage{caption}
\usepackage{subcaption}

\definecolor{myRed}{HTML}{E34A33}
\definecolor{myBlue}{HTML}{0571B0}
\definecolor{myBrown}{HTML}{A6611A}

\definecolor{linkBlue}{HTML}{0055C9}
\definecolor{linkRed}{HTML}{FF1A24}
\definecolor{linkPurple}{HTML}{6200D9}

\usepackage[%
  colorlinks = true,
  linkcolor  = linkBlue,
  citecolor  = linkRed,
  urlcolor   = linkPurple
]{hyperref}
\usepackage{url}
\usepackage[nameinlink]{cleveref}

\crefformat{equation}{\textup{#2(#1)#3}}
\crefrangeformat{equation}{\textup{#3(#1)#4--#5(#2)#6}}
\crefmultiformat{equation}{\textup{#2(#1)#3}}{ and \textup{#2(#1)#3}}
{, \textup{#2(#1)#3}}{, and \textup{#2(#1)#3}}
\crefrangemultiformat{equation}{\textup{#3(#1)#4--#5(#2)#6}}%
{ and \textup{#3(#1)#4--#5(#2)#6}}{, \textup{#3(#1)#4--#5(#2)#6}}
{, and \textup{#3(#1)#4--#5(#2)#6}}

\usepackage{enumitem}

\usepackage{todonotes}

\newcommand{\M}{\ensuremath{\mathbb{M}}}
\newcommand{\R}{\ensuremath{\mathbb{R}}}
\newcommand{\C}{\ensuremath{\mathbb{C}}}

\newcommand{\cB}{\ensuremath{\mathcal{B}}}
\newcommand{\cC}{\ensuremath{\mathcal{C}}}
\newcommand{\cK}{\ensuremath{\mathcal{K}}}
\newcommand{\cN}{\ensuremath{\mathcal{N}}}
\newcommand{\cH}{\ensuremath{\mathcal{H}}}
\newcommand{\tcN}{\ensuremath{\skew{2}\widetilde{\cN}}}

\newcommand{\cBr}{\ensuremath{\skew{4}\widehat{\cB}}}
\newcommand{\cCr}{\ensuremath{\skew{4}\widehat{\cC}}}
\newcommand{\cKr}{\ensuremath{\skew{4}\widehat{\cK}}}
\newcommand{\cNr}{\ensuremath{\skew{2}\widehat{\cN}}}
\newcommand{\cHr}{\ensuremath{\skew{4}\widehat{\cH}}}
\newcommand{\Gr}{\ensuremath{\widehat{G}}}
\newcommand{\mur}{\ensuremath{\skew{2}\hat{\mu}}}
\newcommand{\yr}{\ensuremath{\hat{y}}}
\newcommand{\xr}{\ensuremath{\hat{x}}}

\newcommand{\Mr}{\ensuremath{\widehat{M}}}
\newcommand{\Dr}{\ensuremath{\widehat{D}}}
\newcommand{\Kr}{\ensuremath{\widehat{K}}}
\newcommand{\Nr}{\ensuremath{\widehat{N}}}
\newcommand{\Bur}{\ensuremath{\widehat{B}_{\mathrm{u}}}}
\newcommand{\Cpr}{\ensuremath{\widehat{C}_{\mathrm{p}}}}
\newcommand{\Cvr}{\ensuremath{\widehat{C}_{\mathrm{v}}}}

\newcommand{\trans}{\ensuremath{\mkern-1.5mu\mathsf{T}}}
\newcommand{\herm}{\ensuremath{\mathsf{H}}}

\DeclareMathOperator{\mspan}{span}
\DeclareMathOperator{\err}{err}
\DeclareMathOperator{\mtime}{t}
\DeclareMathOperator{\mfreq}{f}
\DeclareMathOperator{\td}{1}
\DeclareMathOperator{\msd}{2}

\theoremstyle{plain}\newtheorem{theorem}{Theorem}
\theoremstyle{plain}
\theoremstyle{definition}\newtheorem{remark}{Remark}

\begin{document}

\title{Structure-Preserving Interpolation for Model Reduction of %
  Parametric Bilinear Systems}
  
\author{%
  Peter~Benner\thanks{
   Max Planck Institute for Dynamics of Complex Technical Systems,
   Sandtorstr. 1, 39106 Magdeburg, Germany.\newline
   E-mail: \texttt{\href{mailto:benner@mpi-magdeburg.mpg.de}%
     {benner@mpi-magdeburg.mpg.de}}
   \newline
   Otto von Guericke University, Faculty of Mathematics,
   Universit{\"a}tsplatz 2, 39106 Magdeburg, Germany.\newline
   E-mail: \texttt{\href{mailto:peter.benner@ovgu.de}%
     {peter.benner@ovgu.de}}\newline
   ORCID: \texttt{\href{https://orcid.org/0000-0003-3362-4103}%
     {0000-0003-3362-4103}}} \and
  Serkan~Gugercin\thanks{
 Department of Mathematics  and
Computational Modeling and Data Analytics Division,
Academy of Integrated Science,
Virginia Tech, Blacksburg, VA 24061, USA.\newline
   E-mail: \texttt{\href{mailto:gugercin@vt.edu}%
     {gugercin@vt.edu}}\newline
   ORCID: \texttt{\href{https://orcid.org/0000-0003-4564-5999}%
     {0000-0003-4564-5999}}} \and
  Steffen~W.~R.~Werner\thanks{
    Max Planck Institute for Dynamics of Complex
    Technical Systems, Sandtorstr. 1, 39106 Magdeburg, Germany.\newline
    E-mail: \texttt{\href{mailto:werner@mpi-magdeburg.mpg.de}%
      {werner@mpi-magdeburg.mpg.de}}\newline
   ORCID: \texttt{\href{https://orcid.org/0000-0003-1667-4862}%
     {0000-0003-1667-4862}}}
}

\date{~}


\makeatletter
\twocolumn[
  \begin{@twocolumnfalse}
    \maketitle
    \vspace{-2\baselineskip}
    \begin{abstract}
      In this paper, we present an interpolation framework for  
      structure-preserving model order reduction of parametric bilinear
      dynamical systems.
      We introduce a general setting, covering a broad variety of different 
      structures for parametric bilinear systems, and then provide conditions
      on projection spaces for the interpolation of structured subsystem
      transfer functions such that the system structure and parameter
      dependencies are preserved in the reduced-order model.
      Two benchmark examples with different parameter dependencies are used to 
      demonstrate the theoretical analysis.
      
      \vspace{1em}
    \end{abstract}
    \vspace{\baselineskip}
  \end{@twocolumnfalse}
]{
  \renewcommand{\thefootnote}%
    {\fnsymbol{footnote}}
  \footnotetext[1]{%
    Max Planck Institute for Dynamics of Complex Technical Systems,
     Sandtorstr. 1, 39106 Magdeburg, Germany.\newline
     E-mail: \texttt{\href{mailto:benner@mpi-magdeburg.mpg.de}%
       {benner@mpi-magdeburg.mpg.de}}
     \newline
     Otto von Guericke University, Faculty of Mathematics,
     Universit{\"a}tsplatz 2, 39106 Magdeburg, Germany.\newline
     E-mail: \texttt{\href{mailto:peter.benner@ovgu.de}%
       {peter.benner@ovgu.de}}\newline
     ORCID: \texttt{\href{https://orcid.org/0000-0003-3362-4103}%
       {0000-0003-3362-4103}}}
  \renewcommand{\thefootnote}%
    {\fnsymbol{footnote}}
  \footnotetext[2]{%
    Department of Mathematics and Computational Modeling and Data Analytics
    Division, Academy of Integrated Science, Virginia Tech, Blacksburg,
    VA 24061, USA.\newline
    E-mail: \texttt{\href{mailto:gugercin@vt.edu}%
      {gugercin@vt.edu}}\newline
    ORCID: \texttt{\href{https://orcid.org/0000-0003-4564-5999}%
      {0000-0003-4564-5999}}}
  \renewcommand{\thefootnote}%
    {\fnsymbol{footnote}}
  \footnotetext[3]{%
    Max Planck Institute for Dynamics of Complex
    Technical Systems, Sandtorstr. 1, 39106 Magdeburg, Germany.\newline
    E-mail: \texttt{\href{mailto:werner@mpi-magdeburg.mpg.de}%
      {werner@mpi-magdeburg.mpg.de}}\newline
    ORCID: \texttt{\href{https://orcid.org/0000-0003-1667-4862}%
      {0000-0003-1667-4862}}}
}
\makeatother


\section{Introduction}%
\label{sec:intro}

Design and control processes usually involve simulating systems of differential
equations describing the underlying dynamics.
In the setting of nonlinear or stochastic processes,
an important class of such systems are parametric bilinear time-invariant
systems; see, e.g.,~\cite{morAlbFB93, Moh73, Ou10, SapSH19} for some
applications of bilinear systems.
In most cases, these bilinear systems  have special 
structures resulting from the underlying physical model and the dynamics are
parameter dependent. For example, in case of parametric bilinear mechanical
systems, they have the form 
\begingroup
\small
\begin{align} \label{eqn:bsosys}
  \begin{aligned}
  & M(\mu) \ddot{x}(t;\mu) + D(\mu) \dot{x}(t;\mu) + K(\mu) x(t;\mu) =
    B_{\mathrm{u}}(\mu)u(t) \\
  & \quad{}+{} \sum\limits_{j = 1}^{m} N_{\mathrm{p},j}(\mu) x(t;\mu)
      u_{j}(t)+{}\sum\limits_{j = 1}^{m}N_{\mathrm{v},j}(\mu) \dot{x}(t;\mu)
      u_{j}(t),\\
  & y(t;\mu) = C_{\mathrm{p}}(\mu) x(t;\mu) + C_{\mathrm{v}}(\mu) 
      \dot{x}(t;\mu),
  \end{aligned}
\end{align}
\endgroup
where $M(\mu)$, $D(\mu)$, $K(\mu)$, $N_{\mathrm{p}, j}(\mu)$,
$N_{\mathrm{v}, j}(\mu) \in \R^{n \times n}$, for $j = 1, \ldots, m$;
$B_{\mathrm{u}}(\mu) \in \R^{n \times m}$ and
$C_{\mathrm{p}}(\mu), C_{\mathrm{v}}(\mu) \in \R^{p \times n}$ are constant
matrices; and $\mu \in \M \subset \R^{d}$ represents the (constant) parameters
affecting the dynamics.
In~\cref{eqn:bsosys},
\begin{align*}
  u(t) = \begin{bmatrix} u_{1}(t), & u_{2}(t), & \ldots, & u_{m}(t)
    \end{bmatrix}^{\trans} \in \R^{m}
\end{align*}
denotes the inputs (forcing), $y(t;\mu)\in \R^p$ the outputs (measurements),
and ${x}(t;\mu)  \in \R^{n \times n}$  the internal variables. The parameter
$\mu$ may represent variations in,  e.g., material properties
or system geometry.

Due to an increasing demand for accuracy in the modeling stage, systems as
in~\cref{eqn:bsosys} become larger and larger, e.g., $n>10^6$,  imposing
overwhelming demands on computational resources like time and memory.
The situation is even more prominent in the parametric problems we consider here 
due to the need to evaluate/simulate~\cref{eqn:bsosys} for many samples of
$\mu$.
The aim of parametric model order reduction is to construct a cheap-to-evaluate 
approximation of the input-to-output behavior of the original system by reducing 
the state-space dimension, i.e., the number of equations $n$, in such a way that
the reduced model provides a high-fidelity approximation to the original one for
the parameter range of interest.
Additionally, the reduced-order model should have the same internal structure
as well as the parameter dependencies as the original to retain the underlying
physical structure.
For example, for the system~\cref{eqn:bsosys}, the structure-preserving
parametric reduced-order model will have the form
\begingroup
\small
\begin{align} \label{eqn:bsosysred}
  \begin{aligned}
    & \Mr(\mu) \ddot{\xr}(t;\mu) + \Dr(\mu) \dot{\xr}(t;\mu)
      + \Kr(\mu) \xr(t;\mu) =\Bur(\mu) u(t)\\
    & \quad{}+{} \sum\limits_{j = 1}^{m} \Nr_{\mathrm{p},j}(\mu) \xr(t;\mu)
      u_{j}(t) + \sum\limits_{j = 1}^{m} \Nr_{\mathrm{v},j}(\mu) 
      \dot{\xr}(t;\mu) u_{j}(t),\\
    & \yr(t;\mu) = \Cpr(\mu) \xr(t;\mu) + \Cvr(\mu) \dot{\xr}(t;\mu),
  \end{aligned}
\end{align}
\endgroup
with $\Mr(\mu), \Dr(\mu), \Kr(\mu), \Nr_{\mathrm{p}, j}(\mu),
\Nr_{\mathrm{v}, j}(\mu) \in \R^{r \times r}$, for $j = 1, \ldots, m$,
$\Bur(\mu) \in \R^{r \times m}$, $\Cpr(\mu), \Cvr(\mu) \in \R^{p \times r}$,
and $r \ll n$.
Note that the reduced-order model~\cref{eqn:bsosysred} has the same structure
as~\cref{eqn:bsosys} and can be interpreted as a 
physically meaningful reduced-order mechanical system. 
The structure preservation can also be very beneficial in terms of computational 
speed and accuracy; see, e.g.,~\cite{morBenGW20, morBenKS13}.

For parametric \emph{unstructured (classical)} bilinear systems, i.e.,
for systems of the form
\begingroup
\small
\begin{align} \label{eqn:bsys}
 \begin{aligned}
   E(\mu) \dot{x}(t; \mu) & = A(\mu) x(t; \mu) + B(\mu) u(t) \\
   & \quad{}+{} \sum\limits_{j = 1}^{m} N_{j}(\mu) x(t; \mu) u_{j}(t), \\
   y(t; \mu) & = C(\mu) x(t; \mu),
 \end{aligned}
\end{align}
\endgroup
the interpolatory parametric model reduction framework
was developed in~\cite{morRodGB18} by synthesizing the interpolation theory for 
parametric linear dynamical systems~\cite{morAntBG20, morBauBBetal11} with the 
subsystem interpolation approaches for bilinear
systems~\cite{morAntBG20, morBaiS06, morBreD10, morConI07, morFenB07a}.
Recently in~\cite{morBenGW20}, the structured interpolation framework
of~\cite{morBeaG09} for linear dynamical systems has been
extended to the case of structured bilinear systems for \emph{non-parametric}
structured bilinear systems.
In this paper, we will extend this interpolation theory to the
case of \emph{structured parametric} bilinear systems.

In \Cref{sec:basics}, we introduce the basic mathematical concepts and notation. 
We prove the structure-preserving interpolation framework
for parametric bilinear systems in \Cref{sec:strint}.
The established theory is then extended in \Cref{sec:paramsens} to the 
interpolation of parameter sensitivities.
\Cref{sec:examples} illustrates the analysis in two numerical benchmark
examples, followed by conclusions in \Cref{sec:conclusions}.


\section{Mathematical preliminaries}%
\label{sec:basics}

Under some mild assumptions, the output of the bilinear 
system~\cref{eqn:bsys} can be rewritten in terms of a Volterra 
series, i.e.,
\begingroup
\small
\begin{align*}
  y(t; \mu) & = \sum\limits_{k = 1}^{\infty} \int\limits_{0}^{t}
    \int\limits_{0}^{t_{1}} \ldots \int\limits_{0}^{t_{k-1}}
    g_{k}(t_{1}, \ldots, t_{k}, \mu) \\
  & \quad{}\times{}  \left( u(t - \sum\limits_{i = 1}^{k}t_{i}) \otimes \cdots 
    \otimes u(t - t_{1}) \right) \mathrm{d}t_{k} \cdots \mathrm{d}t_{1},
\end{align*}
\endgroup
where $g_{k}$ denotes the $k$-th regular Volterra kernel; see,
e.g., \cite{Rug81}.
Using the multivariate Laplace transformation~\cite{Rug81}, the regular Volterra 
kernels yield the frequency representation~\cref{eqn:bsystf},
as the $k$-th regular transfer function of~\cref{eqn:bsys}, where
$N(\mu) = \begin{bmatrix} N_{1}(\mu), & \ldots, & N_{m}(\mu) \end{bmatrix}$.
The model reduction theory in~\cite{morRodGB18} is based on the interpolation 
of~\cref{eqn:bsystf}, i.e., unstructured (classical) parametric subsystems.

\begin{figure*}
\begingroup
\small
\begin{align} \label{eqn:bsystf}
  G_{k}(s_{1}, \ldots, s_{k}, \mu) & = C(\mu) (s_{k}E(\mu) - A(\mu))^{-1} 
    \left( \prod\limits_{j = 1}^{k-1} (I_{m^{j-1}} \otimes N(\mu))
    (I_{m^{j}} \otimes (s_{k-j}E(\mu) - A(\mu))^{-1}) \right)
    (I_{m^{k-1}} \otimes B(\mu)),~~k \geq 1 \\ \label{eqn:tfmimo}
  G_{k}(s_{1}, \ldots, s_{k}, \mu) & = \cC(s_{k}, \mu)\cK(s_{k}, \mu)^{-1}
    \left(\prod\limits_{j = 1}^{k-1} (I_{m^{j-1}} \otimes \cN(s_{k-j}, \mu))
    (I_{m^{j}} \otimes \cK(s_{k-j}, \mu)^{-1}) \right)
    (I_{m^{k-1}} \otimes \cB(s_{1}, \mu)),~~k \geq 1 
    \\ \label{eqn:tfmimored}
  \Gr_{k}(s_{1}, \ldots, s_{k}, \mu) & = \cCr(s_{k}, \mu)\cKr(s_{k}, \mu)^{-1}
    \left(\prod\limits_{j = 1}^{k-1} (I_{m^{j-1}} \otimes \cNr(s_{k-j}, \mu))
    (I_{m^{j}} \otimes \cKr(s_{k-j}, \mu)^{-1}) \right)
    (I_{m^{k-1}} \otimes \cBr(s_{1}, \mu)),~~k \geq 1 
\end{align}
\endgroup
\hrule
\end{figure*}

In this paper, we consider a much more general setting of multivariate transfer 
functions.
The interpolation of structured transfer functions for linear systems was
developed in \cite{morBeaG09} and then extended to the parametric setting
in~\cite{morAntBG10}.
As the structured transfer functions were recently extended to non-parametric
bilinear systems in~\cite{morBenGW20}, we  consider here
\emph{structured parametric multivariate transfer functions} of the 
form \cref{eqn:tfmimo} with frequency points
$s_{1},$ $\ldots,$ $s_{k} \in \C$, constant parameters
$\mu \in \M \subset \R^{d}$,  
$\cN(s, \mu) = \begin{bmatrix} \cN_{1}(s, \mu), & \ldots, & 
\cN_{m}(s, \mu) \end{bmatrix}$,
and matrix functions
\begin{align*}
  \begin{aligned}
    \cC\colon \C \times \M \rightarrow \C^{p \times n}, &&
    \cK\colon \C \times \M \rightarrow \C^{n \times n}, \\
    \cB\colon \C \times \M \rightarrow \C^{n \times m}, &&
    \cN_{j}\colon \C \times \M \rightarrow \C^{n \times n},
  \end{aligned}
\end{align*}
for $j = 1, \ldots, m$.
For the  parametric bilinear mechanical
systems~\cref{eqn:bsosys}, these matrix functions are realized by
\begin{align*}
  \begin{aligned}
    \cK(s, \mu) & = s^{2} M(\mu) + s D(\mu) + K(\mu),\\
      \cN_{j}(s, \mu) & = N_{\mathrm{p}, j}(\mu) + s N_{\mathrm{v}, j}(\mu)
      ~\text{for}~j = 1, \ldots, m, \\
    \cB(s, \mu) & = B_{\mathrm{u}},~\text{and}~\cC(s, \mu) = 
      C_{\mathrm{p}}(\mu) + s C_{\mathrm{v}}(\mu).
  \end{aligned}
\end{align*}
The reduced-order models are then computed by projection: given model reduction 
bases $V, W \in \C^{n \times r}$, the reduced-order model $\Gr$ is described 
by the reduced-order matrix functions
\begingroup
\small
\begin{align} \label{eqn:project}
  \begin{aligned}
    \cCr(s, \mu) & = \cC(s, \mu) V, &
      \cKr(s, \mu) & = W^{\herm} \cK(s, \mu) V, \\
    \cBr(s, \mu) & = W^{\herm} \cB(s, \mu), &
      \cNr_{j}(s, \mu) & = W^{\herm} \cN_{j}(s, \mu) V,
  \end{aligned}
\end{align}
\endgroup
for $j = 1, \ldots, m$.
In general, every matrix-valued function can be affinely decomposed with respect 
to its arguments and we can write
\begin{align*}
  \cK(s, \mu) = \sum\limits_{j = 1}^{n_{\cK}} h_{\cK, j}(s, \mu) \cK_{j},
\end{align*}
where $h_{\cK, j}\colon \C \times \M \rightarrow \C$ are scalar
functions depending on frequency and parameter, and $\cK_{j} \in
\C^{n \times n}$ are constant matrices, for $j = 1, \ldots, n_{\cK}$.
In the worst-case scenario, we have $n_{\cK} = n^2$ and $\cK_j$'s are the
elementary matrices.
However, we are interested in cases where $n_{\cK}$ is modest, which is the
case in most applications. 
Using the affine decomposition, the reduced-order matrix function is then given
by
\begin{align*}
  \begin{aligned}
    \cKr(s, \mu) & = W^{\herm} \cK (s, \mu) V =
      \sum\limits_{j = 1}^{n_{\cK}} h_{\cK, j}(s, \mu) W^{\herm} \cK_{j} V\\
    &  = \sum\limits_{j = 1}^{n_{\cK}} h_{\cK, j}(s, \mu) \cKr_{j}.
  \end{aligned}
\end{align*}
This works analogously for the other matrix functions in~\cref{eqn:project},
which gives a computable realization of the reduced-order model. Since the
functions $h_{\cK, j}$ stay unchanged, the internal structure and 
parameter dependency of the original matrix functions, (and thus of
the original system) are retained.

In the following, we will use an abbreviation for the notion of partial 
derivatives, namely we denote
\begin{align*}
  \partial_{s_{1}^{j_{1}} \cdots s_{k}^{j_{k}}} f(z_{1}, \ldots, z_{k}) & :=
    \frac{\partial^{j_{1} + \ldots + j_{k}} f}{\partial s_{1}^{j_{1}} \cdots
    \partial s_{k}^{j_{k}}} (t_{1}, \ldots, t_{k}),
\end{align*}
for the differentiation of an analytic function $f\colon \C^{k}
\rightarrow \C^{\ell}$ with respect to the variables $s_{1}, \ldots,
s_{k}$ and evaluated at $z_{1},$ $\ldots,$ $z_{k}$.
Also, we denote the vertical concatenation of the bilinear terms by
$\tcN(s, \mu)  = \begin{bmatrix} \cN_{1}(s, \mu) \\ \vdots \\
\cN_{m}(s, \mu) \end{bmatrix}$.


\section{Structured interpolation}%
\label{sec:strint}

Interpolatory model reduction has been one of the most commonly used and
effective approaches to model reduction and shown to provide locally optimal
reduced models for linear, bilinear, quadratic-bilinear dynamical systems;
we refer the reader to \cite{morAntBG20, morBauBF14, morScaA17} and references
therein for details on interpolatory model reduction for linear and
nonlinear systems.
In this setting, one chooses  $V$ and $W$ in~\cref{eqn:project} such that the 
re\-du\-ced-or\-der transfer functions interpolate the transfer functions of the 
original system at selected points.
In the setting of parametric structured multivariate transfer functions $G_k$
in~\cref{eqn:tfmimo}, we want to construct $V$ and $W$ such that the reduced
transfer functions $\Gr_k$ in~\cref{eqn:tfmimored} satisfy
\begin{align} \label{eqn:lag}
  G_{k}(\sigma_{1}, \ldots, \sigma_{k}, \mur) &
    = \Gr_{k}(\sigma_{1}, \ldots, \sigma_{k}, \mur)~~\text{and} \\
  \label{eqn:der}
  \nabla G_{k}(\sigma_{1}, \ldots, \sigma_{k}, \mur) &
    = \nabla \Gr_{k}(\sigma_{1}, \ldots, \sigma_{k}, \mur)
\end{align}
for given frequency interpolation points $\sigma_{1},$ $\ldots,$ 
$\sigma_{k} \in \C$ and the parameter interpolation point $\mu \in \M$
where $\nabla G_{k}$ denotes the Jacobian matrix
\begin{align*}
  \nabla G_{k} = \begin{bmatrix} \partial_{s_{1}} G_{k}, & \ldots, & 
    \partial_{s_{k}} G_{k}, & \partial_{\mu_{1}} G_{k}, & \ldots, & 
    \partial_{\mu_{d}} G_{k} \end{bmatrix}.
\end{align*}
We emphasize  that for multi-input/multi-output (MIMO) systems we consider here,
transfer functions $G_k$ are matrix valued.
Therefore, conditions in~\cref{eqn:lag}~and~\cref{eqn:der} enforce matrix
interpolation.
This is not usually needed.
For MIMO linear dynamical systems, for example, one enforces tangential
interpolation, meaning matrix-interpolation along selected
directions~\cite{morAntBG20}.
However, for brevity and to keep  the notation concise, we will focus on
matrix interpolation. 

Even though we have only listed two sets of  interpolation conditions
in~\cref{eqn:lag,eqn:der}, \Cref{thm:mtxvw,thm:mtxvwhermite}
below will show how to construct $V$ and $W$ to enforce interpolation for more 
general cases, including higher-order partial derivatives.
The recent work~\cite{morBenGW20} showed how to enforce~\cref{eqn:lag,eqn:der}
for \emph{non-parametric} structured bilinear systems.
Our theory below will extend these results to the parametric case.
Note that the first condition~\cref{eqn:lag} does not involve any
differentiation with respect to the parameter $\mur$ and can be viewed as
interpolation for a fixed parameter $\mu = \mur$.
Therefore, we might expect that the subspace constructions
from~\cite{morBenGW20} for the non-parametric problem might yield the desired
subspaces. 
This is indeed what we discuss first in \Cref{thm:mtxvw,thm:mtxvwhermite}.
However, the second condition~\cref{eqn:der} involves matching sensitivity with
respect to the parameter as well, which will be discussed in
\Cref{sec:paramsens}.

\begin{theorem}[Structured matrix interpolation]%
  \label{thm:mtxvw}
  Let $G$ be a parametric bilinear system, with its structured subsystem
  transfer functions $G_k$ in~\cref{eqn:tfmimo}, and  $\Gr$ be the
  re\-du\-ced-order parametric bilinear system, constructed as
  in~\cref{eqn:project} with its  subsystem transfer functions $\Gr_k$
  in~\cref{eqn:tfmimored}.
  Let  the matrix functions 
  $\cC(s, \mu)$, $\cK(s, \mu)^{-1}$, $\cN(s, \mu)$, $\cB(s, \mu)$, and
  $\cKr(s, \mu)^{-1}$ be defined for given sets of frequency interpolation
  points $\sigma_{1}, \ldots, \sigma_{k} \in \C$ and $\varsigma_{1}, \ldots,
  \varsigma_{\theta} \in \C$, and the parameter interpolation point
  $\mur \in \M$.
  \begin{enumerate}[label = (\alph*)]
    \item If $V$ is constructed such that
      \begin{align*}
        \mspan(V) \supseteq \mspan([V_{1}, \ldots, V_{k}]),
      \end{align*}
      where
      \begin{align} \label{eqn:V1Thm1}
        \begin{aligned}
          V_{1} & = \cK(\sigma_{1}, \mur)^{-1}
            \cB(\sigma_{1}, \mur)~\text{and}\\
          V_{j} & = \cK(\sigma_{j}, \mur)^{-1} \cN(\sigma_{j-1}, \mur)
            (I_{m} \otimes V_{j-1}),
        \end{aligned}
      \end{align}
      for $2 \leq j \leq k$,
      then the following interpolation conditions hold true:
      \begin{align}  \label{eqn:conda}
        G_{j}(\sigma_{1}, \ldots, \sigma_{j}, \mur)
          = \Gr_{j}(\sigma_{1}, \ldots, \sigma_{j}, \mur),
      \end{align}
      for $j = 1, \ldots, k$.
    \item If $W$ is constructed such that
      \begin{align*}
        \mspan(W) \supseteq \mspan([W_{1}, \ldots, W_{\theta}]),
      \end{align*}
      where
      \begin{align*}
        \begin{aligned}
          W_{1} & = \cK(\varsigma_{\theta}, \mur)^{-\herm}
            \cC(\varsigma_{\theta}, \mur)^{\herm}~\text{and}\\
          W_{i} & = \cK(\varsigma_{\theta-i+1}, \mur)^{-\herm}
            \tcN(\varsigma_{\theta-i+1}, \mur)^{\herm} (I_{m} \otimes W_{i-1}),
        \end{aligned}
      \end{align*}
      for $2 \leq i \leq \theta$, then the following interpolation conditions 
      hold true:
      \begin{align} \label{eqn:condb}
        G_{i}(\varsigma_{\theta-i+1}, \ldots, \varsigma_{\theta}, \mur)
          = \Gr_{i}(\varsigma_{\theta-i+1}, \ldots, \varsigma_{\theta}, \mur),
      \end{align}
      for $i = 1, \ldots, \theta$.
    \item Let $V$ be constructed as in Part (a) and $W$ as in Part (b).
      Then, in addition to~\cref{eqn:conda}~and~\cref{eqn:condb}, the 
      interpolation conditions
      \begin{align} \label{eqn:condc}
        \begin{aligned}
          & G_{q + \eta}(\sigma_{1}, \ldots, \sigma_{q},
            \varsigma_{\theta-\eta+1}, \ldots, \varsigma_{\theta}, \mur) \\
          & = \Gr_{q + \eta}(\sigma_{1}, \ldots, \sigma_{q},
            \varsigma_{\theta-\eta+1}, \ldots, \varsigma_{\theta}, \mur), 
        \end{aligned}
      \end{align}
      hold for $1 \leq q \leq k$ and  $1 \leq \eta \leq \theta$.
  \end{enumerate}
\end{theorem}
\begin{proof} \allowdisplaybreaks
  Given the fixed parameter $\mur \in \M$, the matrix functions $\cC(s, \mur)$,
  $\cK(s, \mur)$, $\cN(s, \mur)$ and $\cB(s, \mur)$ can be viewed as
  the realization of a non-pa\-ra\-me\-tric  bilinear system.
  Then, the interpolation conditions~\cref{eqn:conda,eqn:condb,eqn:condc} can be 
  considered as subsystem interpolation of a non-parametric bilinear system as 
  these conditions do not involve any variation/sensitivity with respect to 
  $\mu$.
  Therefore, the subspace conditions in~\cite[Theorem~8]{morBenGW20}, for
  interpolating a non-parametric structured bilinear system, apply here as
  well, which are precisely the subspace conditions listed in Parts (a)--(c). 
  However, to make the paper self-contained  and the proof of 
  \Cref{thm:mtxvwtwo} in \Cref{sec:paramsens} easier to follow, we will still 
  prove Part~(a) for $k = 2$.
  By induction over $k$, the rest of the result in (a) follows directly using the
  same arguments.
  Using~\cref{eqn:tfmimored}, the second reduced-order transfer 
  function is given by
  \begin{align*}
    \Gr_{2}(\sigma_{1}, \sigma_{2}, \mur) & = \cCr(\sigma_{2}, \mur)
      \cKr(\sigma_{2}, \mur)^{-1} \cNr(\sigma_{1}, \mur)\\
    & \quad{}\times{} (I_{m} \otimes \cKr(\sigma_{1}, \mur)^{-1})
      (I_{m} \otimes \cBr(\sigma_{1}, \mur)).
  \end{align*}
  We observe that with~\cref{eqn:project} it holds
  \begin{align*}
    & (I_{m} \otimes V) (I_{m} \otimes \cKr(\sigma_{1}, \mur)^{-1})
      (I_{m} \otimes \cBr(\sigma_{1}, \mur))\\
    & = (I_{m} \otimes V \cKr(\sigma_{1}, \mur)^{-1} \cBr(\sigma_{1}, \mur))\\
    & = (I_{m} \otimes V \cKr(\sigma_{1}, \mur)^{-1} W^{\herm}
      \cB(\sigma_{1}, \mur))\\
    & = (I_{m} \otimes \underbrace{V \cKr(\sigma_{1}, \mur)^{-1} W^{\herm}
      \cK(\sigma_{1}, \mur)}_{P_{\mathrm{V}_{1}}}
      \underbrace{\cK(\sigma_{1}, \mur)^{-1} \cB(\sigma_{1}, \mur)}_{V_{1}}),
  \end{align*}
  where $P_{\mathrm{V}_{1}}$ is a projector onto $\mspan(V)$ and $V_{1}$ is as 
  defined in \cref{eqn:V1Thm1}.
  By construction, we have $\mspan(V_{1}) \subseteq \mspan(V)$; thus 
  $P_{\mathrm{V}_{1}} V_{1} = V_{1}$ and, therefore
  \begin{align*}
    & (I_{m} \otimes V) (I_{m} \otimes \cKr(\sigma_{1}, \mur)^{-1})
      (I_{m} \otimes \cBr(\sigma_{1}, \mur))\\
    & = (I_{m} \otimes \cK(\sigma_{1}, \mur)^{-1})
      (I_{m} \otimes \cB(\sigma_{1}, \mur)).
  \end{align*}
  Then, $\Gr_{2}$ can be written as
  \begin{align*}
    \Gr_{2}(\sigma_{1}, \sigma_{2}, \mur) & = \cCr(\sigma_{2}, \mur)
      \cKr(\sigma_{2}, \mur)^{-1} W^{\herm} \cN(\sigma_{1}, \mur)\\
    & \quad{}\times{} (I_{m} \otimes \cK(\sigma_{1}, \mur)^{-1})
      (I_{m} \otimes \cB(\sigma_{1}, \mur))\\
    & = \cC(\sigma_{2}, \mur) V \cKr(\sigma_{2}, \mur)^{-1} W^{\herm} 
      \cN(\sigma_{1}, \mur)\\
    & \quad{}\times{} (I_{m} \otimes V_{1}).
  \end{align*}
  Also, it holds that
  \begin{align*}
    & V \cKr(\sigma_{2}, \mur)^{-1} W^{\herm} \cN(\sigma_{1}, \mur)
      (I_{m} \otimes V_{1})\\
    & = \underbrace{V \cKr(\sigma_{2}, \mur)^{-1} W^{\herm}
      \cK(\sigma_{2}, \mur)}_{P_{\mathrm{V}_{2}}} \\
    & \quad{}\times{} \underbrace{\cK(\sigma_{2}, \mur)^{-1}
      \cN(\sigma_{1}, \mur) (I_{m} \otimes V_{1})}_{V_{2}}\\
    & = \cK(\sigma_{2}, \mur)^{-1} \cN(\sigma_{1}, \mur) (I_{m} \otimes V_{1}),
  \end{align*}
  using the fact that $P_{\mathrm{V}_{2}}$ is another projector onto $\mspan(V)$ 
  and that $\mspan(V_{2}) \subseteq \mspan(V)$.
  Inserting this last equality into the second reduced-order transfer function
  yields
  \begin{align*}
    \Gr_{2}(\sigma_{1}, \sigma_{2}, \mur) = G_{2}(\sigma_{1}, \sigma_{2}, \mur).
  \end{align*}
  Constructing further projectors onto $\mspan(V)$ for higher-or\-der transfer
  functions gives the result in (a).
  The result in Part~(b) follows exactly the same way by using the Hermitian 
  transposed matrix functions and constructing now projectors onto $\mspan(W)$.
  Part~(c) is then resulting from the application of both types of projectors
  onto $\mspan(V)$ and $\mspan(W)$.
\end{proof}

In \Cref{thm:mtxvw}, only function values are matched, i.e., the zeroth
derivative.
The following theorem extends these results to matching higher-order
derivatives in the frequency arguments, i.e., to enforcing Hermite
interpolation conditions.

\begin{theorem}[Hermite matrix interpolation]%
  \label{thm:mtxvwhermite}
  Let $G$ be a parametric bilinear system, with its structured subsystem
  transfer functions $G_k$ in~\cref{eqn:tfmimo} and  $\Gr$ be the
  re\-du\-ced-order parametric bilinear system, constructed as
  in~\cref{eqn:project} with its  subsystem transfer functions $\Gr_k$
  in~\cref{eqn:tfmimored}.
  Let  the matrix functions 
  $\cC(s, \mu)$, $\cK(s, \mu)^{-1}$, $\cN(s, \mu)$, $\cB(s, \mu)$, and
  $\cKr(s, \mu)^{-1}$ be analytic for given sets of frequency interpolation
  points $\sigma_{1}, \ldots, \sigma_{k} \in \C$ and $\varsigma_{1}, \ldots,
  \varsigma_{\theta} \in \C$, and the parameter interpolation point
  $\mur \in \M$.
  \begin{enumerate}[label=(\alph*)]
    \item If $V$ is constructed such that
      \begin{align*}
        \mspan(V) \supseteq \mspan([V_{1,0}, \ldots, V_{k, \ell_{k}}]),
      \end{align*}
      where
      \begin{align*}
        \begin{aligned}
          V_{1, j_{1}} & = \partial_{s^{j_{1}}} (\cK^{-1} \cB)
            (\sigma_{1}, \mur)~\text{and}\\
          V_{q, j_{q}} & = \partial_{s^{j_{q}}} \cK^{-1} (\sigma_{q}, \mur)\\
          & \quad{}\times{} \left( \prod\limits_{j = 1}^{q-2}
            \partial_{s^{\ell_{q -j}}} \big( (I_{m^{j-1}} \otimes \cN) \right. \\
          & \left. \quad{}\times{} \vphantom{\prod\limits_{j = 1}^{q-2}}
            (I_{m^{j}} \otimes \cK) \big) (\sigma_{q-j}, \mur) \right)\\
          & \quad{}\times{} \partial_{s^{\ell_{1}}}
            \big( (I_{m^{q-2}} \otimes \cN)(I_{m^{q-1}} \otimes \cK) \\
          & \quad{}\times{}
            (I_{m^{q-1}} \otimes \cB) \big) (\sigma_{1}, \mur),
        \end{aligned}
      \end{align*}
      for $2 \leq q \leq k$ and $0 \leq j_{1} \leq \ell_{1}$;
      $0 \leq j_{q} \leq \ell_{q}$, then the
      following interpolation conditions hold true:
      \begin{align} \label{eqn:herma}
        \begin{aligned}
          & \partial_{s_{1}^{\ell_{1}} \cdots s_{q-1}^{\ell_{q-1}}
            s_{q}^{j_{q}}} G_{q} (\sigma_{1}, \ldots, \sigma_{q}, \mur) \\
          & = \partial_{s_{1}^{\ell_{1}} \cdots s_{q-1}^{\ell_{q-1}}
            s_{q}^{j_{q}}} \Gr_{q} (\sigma_{1}, \ldots, \sigma_{q}, \mur),
        \end{aligned}
      \end{align}
      for $q = 1, \ldots, k$ and $j_{q} = 0, \ldots, \ell_{q}$.
    \item If $W$ is constructed such that
      \begin{align*}
        \mspan(W) \supseteq \mspan([W_{1,0}, \ldots, 
          W_{\theta, \nu_{\theta}}]),
      \end{align*}
      where
      \begin{align*}
        \begin{aligned}
          W_{1, i_{\theta}} & = \partial_{s^{i_{\theta}}} (\cK^{-\herm}
            \cC^{\herm}) (\varsigma_{\theta}, \mur)~\text{and}\\
          W_{\eta, i_{\theta-\eta+1}} & = \partial_{s^{i_{\theta-\eta+1}}}
            (\cK^{-\herm} \tcN^{\herm})(\varsigma_{\theta-\eta+1}, \mur)\\
          & \quad{}\times{} \left( \prod\limits_{i = \theta-\eta+2}^{\theta - 1}
            \partial_{s^{\nu_{i}}} (I_{m^{i-1}} \otimes \cK^{-\herm} \right. \\
          & \left. \vphantom{\prod\limits_{i = \theta-\eta+2}^{\theta - 1}}
            \quad{}\times{} \tcN^{\herm})(\varsigma_{i}, \mur) \right)\\
          & \quad{}\times{} \left( I_{m^{\theta-1}} \otimes
            \partial_{s^{\nu_{\theta}}} (\cK^{-\herm} \cC^{\herm})
            (\varsigma_{\theta}, \mur) \right),
        \end{aligned}
      \end{align*}
      for $2 \leq \eta \leq \theta$ and $0 \leq i_{\theta} \leq \nu_{\theta}$;
      $0 \leq i_{\theta - \eta + 1} \leq \nu_{\theta - \eta + 1}$,
      then the following interpolation conditions hold true:
      \begin{align} \label{eqn:hermb}
        \begin{aligned}
          & \partial_{s_{1}^{i_{\theta - \eta + 1}} s_{2}^{\nu_{2}} \cdots
            s_{\theta}^{\nu_{\theta}}} G_{\eta}
            (\varsigma_{\theta - \eta + 1}, \ldots, \varsigma_{\theta}, \mur)\\
          & = \partial_{s_{1}^{i_{\theta - \eta + 1}} s_{2}^{\nu_{2}} \cdots
            s_{\theta}^{\nu_{\theta}}} \Gr_{\eta}
            (\varsigma_{\theta - \eta + 1}, \ldots, \varsigma_{\theta}, \mur),
        \end{aligned}
      \end{align}
      for $\eta = 1, \ldots, \theta$ and $i_{\eta} = 0, \ldots, \nu_{\eta}$.
    \item Let $V$ be constructed as in (a) and $W$ as in (b). Then,
      in addition to~\cref{eqn:herma} and~\cref{eqn:hermb}, the interpolation
      conditions~\cref{eqn:derivates} hold for $j_{q} = 0, \ldots, \ell_{q}$;
      $i_{\theta - \eta + 1} = 0,$ $\ldots,$ $\nu_{\theta - \eta + 1}$;
      $1 \leq q \leq k$ and $1 \leq \eta \leq \theta$.
  \end{enumerate}
\end{theorem}
\begin{proof} \allowdisplaybreaks
  As in \Cref{thm:mtxvw}, all the interpolation conditions are for a fixed
  parameter $\mur \in \M$, i.e., they can be proven using a similar construction 
  of projectors onto suitable subspaces as in \Cref{thm:mtxvw}.
  Therefore, the subspace conditions in~\cite[Theorem~9]{morBenGW20} can be
  applied here, which are precisely the subspace conditions listed in
  \Cref{thm:mtxvwhermite}.
\end{proof}

\begin{figure*}
  \begin{align} \label{eqn:derivates}
    \begin{aligned}
      & \partial_{s_{1}^{\ell_{1}} \cdots s_{q-1}^{\ell_{q-1}}
        s_{q}^{j_{q}} s_{q+1}^{i_{\theta - \eta + 1}}
        s_{q+2}^{\nu_{\theta - \eta + 2}} \cdots s_{q + \eta}^{\nu_{\theta}}}
        G_{q + \eta} (\sigma_{1}, \ldots, \sigma_{q},
        \varsigma_{\theta - \eta + 1}, \ldots, \varsigma_{\theta}, \mur)\\
      & = \partial_{s_{1}^{\ell_{1}} \cdots s_{q-1}^{\ell_{q-1}}
        s_{q}^{j_{q}} s_{q+1}^{i_{\theta - \eta + 1}}
        s_{q+2}^{\nu_{\theta - \eta + 2}} \cdots s_{q + \eta}^{\nu_{\theta}}}
        \Gr_{q + \eta} (\sigma_{1}, \ldots, \sigma_{q},
        \varsigma_{\theta - \eta + 1}, \ldots, \varsigma_{\theta}, \mur),
    \end{aligned}
  \end{align}
  \hrule
\end{figure*}


\section{Matching parameter sensitivities}%
\label{sec:paramsens}

So far, the interpolation conditions enforced did not show variability with
respect to the parameter $\mu$.
Even in the Hermite conditions matched in \Cref{thm:mtxvwhermite}, the matched
derivatives (sensitivities) are with respect to the frequency points.
This enabled us to directly employ the conditions and analysis
from~\cite{morBenGW20}.
However, for parametric systems it is important to match the parameter
sensitivity  with respect to the parameter variation as well.
This is what we establish in the next result, extending 
the similar results from linear dynamics~\cite{morBauBBetal11} and
unstructured bilinear dynamics~\cite{morRodGB18} to the new parametric structured
framework.
An important conclusion is that the parameter sensitivity is matched
implicitly, i.e., without ever explicitly computing it.
This is achieved by using the same set of frequency interpolation points for
$V$ and $W$. 

\begin{theorem}[Two-sided matrix interpolation with identical point sets]%
  \label{thm:mtxvwtwo}
  Let $G$ be a parametric bilinear system, with its structured subsystem
  transfer functions $G_k$ in~\cref{eqn:tfmimo} and $\Gr$ be the
  re\-du\-ced-order parametric bilinear system, constructed as
  in~\cref{eqn:project} with its  subsystem transfer functions $\Gr_k$
  in~\cref{eqn:tfmimored}.
  Let  the matrix functions $\cC(s, \mu)$, $\cK(s, \mu)^{-1}$, $\cN(s, \mu)$,
  $\cB(s, \mu)$, and $\cKr(s, \mu)^{-1}$ be analytic for a given set of
  frequency interpolation points $\sigma_{1}, \ldots, \sigma_{k} \in \C$ and
  the parameter interpolation point $\mur \in \M$.
  \begin{enumerate}[label=(\alph*)]
    \item  Let $V$  be constructed as in \Cref{thm:mtxvw} Part (a)
      and $W$ be constructed as in \Cref{thm:mtxvw} Part (b) with
      $\varsigma_i = \sigma_i$ for $i=1,2,\ldots,k.$ 
      Then, in addition to \cref{eqn:conda,eqn:condb,eqn:condc} it holds
      \begin{align} \label{eqn:imphermparta}
         \nabla G_{k}(\sigma_{1}, \ldots, \sigma_{k}, \mur) 
         & = \nabla \Gr_{k}(\sigma_{1}, \ldots, \sigma_{k}, \mur).
      \end{align}
    \item Let $V$  be constructed as in \Cref{thm:mtxvwhermite} Part (a)
    and $W$ be constructed as in \Cref{thm:mtxvwhermite} Part (b) with
    $\varsigma_i = \sigma_i$ for $i = 1, 2,\ldots, k.$ 
    Then, in addition to \cref{eqn:herma}--\cref{eqn:derivates}, it holds
    \begin{align} \label{eqn:imphermpartb}
        \begin{aligned}
          & \nabla \left( \partial_{s_{1}^{\ell_{1}} \cdots s_{k}^{\ell_{k}}}
            G_{k} (\sigma_{1}, \ldots, \sigma_{k}, \mur) \right)\\
          & = \nabla \left( \partial_{s_{1}^{\ell_{1}} \cdots s_{k}^{\ell_{k}}}
            \Gr_{k} (\sigma_{1}, \ldots, \sigma_{k}, \mur) \right).
        \end{aligned}
      \end{align}
  \end{enumerate}
\end{theorem}
\begin{proof} \allowdisplaybreaks
  For  brevity, we only prove~\cref{eqn:imphermparta}.  The proof 
  of~\cref{eqn:imphermpartb} follows analogously.
  As in the proof of \Cref{thm:mtxvw}, we will construct
  appropriate projectors onto the projection spaces $\mspan(V)$ or $\mspan(W)$.
  In contrast to \Cref{thm:mtxvwhermite}, we  now also interpolate
  the derivative with respect to the parameters.
  Using the product rule, the partial derivative of $\Gr_{k}$ 
  with respect to
  a single parameter entry $\mu_{i}$, for $1 \leq i \leq d$, is given by
  \begin{align} \label{eqn:partdevmu}
    \begin{aligned}
      & \partial_{\mu_{i}} \Gr_{k}(\sigma_{1}, \ldots, \sigma_{k}, \mur)\\
      & = \sum\limits_{\alpha \in \mathbb{A}}
        \left( \partial_{\mu_{i}^{\alpha_{1}}} \cCr (\sigma_{k}, \mur) \right)
        \left( \partial_{\mu_{i}^{\alpha_{2}}} \cKr^{-1}(\sigma_{k}, \mur) 
        \right) \\
      & \quad{}\times{} \left(\prod\limits_{j = 1}^{k-1} (I_{m^{j-1}} \otimes
        \partial_{\mu_{i}^{\alpha_{2j+1}}} \cNr(\sigma_{k-j}, \mur)) \right.\\
      & \left.\vphantom{\prod\limits_{j = 1}^{k-1}} \quad{}\times{}
        (I_{m^{j}} \otimes \partial_{\mu_{i}^{\alpha_{2j+2}}}
        \cKr^{-1}(\sigma_{k-j}, \mur))\right) \\
      & \quad{}\times{} (I_{m^{k-1}} \otimes \partial_{\mu_{i}^{\alpha_{2k+1}}}
        \cBr(\sigma_{1}, \mur)),
    \end{aligned}
  \end{align}
  where $\mathbb{A}$ denotes the set of all columns
  of the identity matrix of size $2k+1$. In other words, \cref{eqn:partdevmu} is 
  a sum of  $2k+1$ terms where each term corresponds to the vector $\alpha$ 
  taking a value from this set of columns. Therefore, in each  term only a
  single matrix function is differentiated. 
  We will show that every single term in the sum~\cref{eqn:partdevmu} matches 
  the same term in the full order model, thus, summed together,  proving the 
  desired interpolation property~\cref{eqn:imphermparta}. 
  Consider, e.g., the second term in~\cref{eqn:partdevmu}, i.e., the term in 
  which $\alpha$ is the second column of the identity matrix: $\alpha = 
  \begin{bmatrix} \alpha_1 & \alpha_2 &  \alpha_3 &\cdots & 
  \alpha_{2k+1}\end{bmatrix}^{\trans} = \begin{bmatrix} 0 & 1 &  0 & \ldots & 0 
  \end{bmatrix}^{\trans}$.
  Denote the corresponding term by $\cHr_{2}$. Then,
  \begin{align*}
    \cHr_{2} & := \cCr (\sigma_{k}, \mur) \left( \partial_{\mu_{i}}
      \cKr^{-1} (\sigma_{k}, \mur)\right) \\
    & \quad{}\times{}
      \left(\prod\limits_{j = 1}^{k-1}
      (I_{m^{j-1}} \otimes \cNr(\sigma_{k-j}, \mur)) \right. \\
    & \left. \vphantom{\prod\limits_{j = 1}^{k-1}} \quad{}\times{}
      (I_{m^{j}} \otimes \cKr(\sigma_{k-j}, \mur)^{-1}) \right)\\
    & \quad{}\times{}
      (I_{m^{k-1}} \otimes \cBr(\sigma_{1}, \mur)).
  \end{align*}
  The derivative of the inverse appearing in $\cHr_{2}$ is given by
  \begin{align*}
    \partial_{\mu_{i}} \cKr^{-1}(\sigma_{k}, \mur) & = 
      -\cKr(\sigma_{k}, \mur)^{-1}
      \Big(\partial_{\mu_{i}} \cKr(\sigma_{k}, \mur)\Big)
      \cKr(\sigma_{k}, \mur)^{-1}.
  \end{align*}
  Therefore, $\cHr_{2}$ can be rewritten as
  \begin{align*}
    \cHr_{2} & = -\cCr(\sigma_{k}, \mur) \cKr(\sigma_{k}, \mur)^{-1}
      \Big(\partial_{\mu_{i}} \cKr(\sigma_{k}, \mur)\Big)
      \cKr(\sigma_{k}, \mur)^{-1}\\
    & \quad{}\times{} \left(\prod\limits_{j = 1}^{k-1}
      (I_{m^{j-1}} \otimes \cNr(\sigma_{k-j}, \mur)) \right. \\
    & \left. \vphantom{\prod\limits_{j = 1}^{k-1}} \quad{}\times{}
      (I_{m^{j}} \otimes \cKr(\sigma_{k-j}, \mur)^{-1}) \right) \\
    & \quad{}\times{} (I_{m^{k-1}} \otimes \cBr(\sigma_{1}, \mur)) \\
    & =: -\widehat{W}_{1}^{\herm}
      \Big(\partial_{\mu_{i}} \cKr(\sigma_{k}, \mur)\Big)
      \widehat{V}_{k}.
  \end{align*}
  Noting that the model reduction space $V$ were constructed as in  
  \Cref{thm:mtxvw},  we obtain
  \begin{align*}
    V \widehat{V}_{k} & = V \cKr(\sigma_{k}, \mur)^{-1}
      \left(\prod\limits_{j = 1}^{k-1}
      (I_{m^{j-1}} \otimes \cNr(\sigma_{k-j}, \mur)) \right. \\
    & \left. \vphantom{\prod\limits_{j = 1}^{k-1}} \quad{}\times{}
      (I_{m^{j}} \otimes \cKr(\sigma_{k-j}, \mur)^{-1}) \right) \\
    & \quad{}\times{} (I_{m^{k-1}} \otimes \cBr(\sigma_{1}, \mur)) \\
    & = \underbrace{V \cKr(\sigma_{k}, \mur)^{-1} W^{\herm}
      \cK(\sigma_{k}, \mur)}_{P_{\mathrm{V}_{k}}}\\
    & \quad{}\times{} \cK(\sigma_{k}, \mur)^{-1}
      \left(\prod\limits_{j = 1}^{k-1}
      (I_{m^{j-1}} \otimes \cN(\sigma_{k-j}, \mur)) \right. \\
    & \left. \vphantom{\prod\limits_{j = 1}^{k-1}} \quad{}\times{}
      (I_{m^{j}} \otimes \cK(\sigma_{k-j}, \mur)^{-1}) \right) \\
    & \quad{}\times{} (I_{m^{k-1}} \otimes \cB(\sigma_{1}, \mur)) \\
    & = P_{\mathrm{V}_{k}} V_{k}\\
    & = V_{k},
  \end{align*}
  where $P_{\mathrm{V}_{k}}$ is a projector onto $\mspan(V)$.
  Similarly, we have
  \begin{align*}
    W \widehat{W}_{1} & = W \cKr(\sigma_{k}, \mur)^{-\herm}
      \cCr(\sigma_{k}, \mur)^{\herm}\\
    & = \underbrace{W \cKr(\sigma_{k}, \mur)^{-\herm} V
      \cK(\sigma_{k}, \mur)^{\herm}}_{P_{\mathrm{W}_{1}}}
      \underbrace{\cK(\sigma_{k}, \mur)^{-\herm}
      \cC(\sigma_{k}, \mur)^{\herm}}_{W_{1}}\\
    & = W_{1},
  \end{align*}
  with $P_{\mathrm{W}_{1}}$ a projector onto $\mspan(W)$.
  Using those two identities, we obtain
  \begin{align*}
    \cHr_{2} & = -\widehat{W}_{1}^{\herm}
      \Big(\partial_{\mu_{i}} \cKr(\sigma_{k}, \mur)\Big)
      \widehat{V}_{k}\\
    & = -\widehat{W}_{1}^{\herm} W^{\herm} \Big(\partial_{\mu_{i}} 
      \cK(\sigma_{k}, \mur)\Big) V \widehat{V}_{k}\\
    & = -W_{1}^{\herm} \Big(\partial_{\mu_{i}} 
      \cK(\sigma_{k}, \mur)\Big) V_{k}\\
    & = \cC (\sigma_{k}, \mur) \left( \partial_{\mu_{i}}
      \cK^{-1} (\sigma_{k}, \mur)\right) \\
    & \quad{}\times{}
      \left(\prod\limits_{j = 1}^{k-1}
      (I_{m^{j-1}} \otimes \cN(\sigma_{k-j}, \mur)) \right. \\
    & \left. \vphantom{\prod\limits_{j = 1}^{k-1}} \quad{}\times{}
      (I_{m^{j}} \otimes \cK(\sigma_{k-j}, \mur)^{-1}) \right)\\
    & \quad{}\times{}
      (I_{m^{k-1}} \otimes \cB(\sigma_{1}, \mur)),
  \end{align*}
  i.e., $\cHr_{2}$ is identical to the term using the original
  matrix functions.
  Since the same technique can be used for all other $\alpha$ values 
  corresponding the other columns in the set $\mathbb{A}$, we obtain,  for all 
  $1 \leq i \leq d$,
  \begin{align} \label{parampartial}
    \partial_{\mu_{i}} \Gr_{k}(\sigma_{1}, \ldots, \sigma_{k}, \mur)
      & = \partial_{\mu_{i}} G_{k}(\sigma_{1}, \ldots, \sigma_{k}, \mur).
  \end{align}
  Interpolation of the partial derivatives with respect to the frequency 
  parameters follows by using the fixed parameter $\mur$ 
  in~\cite[Corollary~2]{morBenGW20}. Together with~\cref{parampartial}, this 
  proves~\cref{eqn:imphermparta}.
\end{proof}

\begin{figure*}
  \begin{align} \label{eqn:sisotf}
    G_{k}(s_{1}, \ldots, s_{k}, \mu) & = \cC(s_{k}, \mu)\cK(s_{k}, \mu)^{-1}
      \left(\prod\limits_{j = 1}^{k-1} \cN(s_{k-j}, \mu)\cK(s_{k-j}, \mu)^{-1})
      \right) \cB(s_{1}, \mu)
  \end{align}
  \hrule
\end{figure*}

\begin{remark}
  \Cref{thm:mtxvwtwo} shows how to match the parameter sensitivity implicitly 
  without ever computing this quantity.
  Matching  the parameter sensitivities is important, especially in the setting 
  of optimization and design.
  These results can be extended to match the parameter Hessian as well; 
  compare to~\cite{morRodGB18}.
  However, we skip those details for brevity.
\end{remark}

\begin{remark}
  All the results in \Cref{thm:mtxvw,thm:mtxvwhermite,thm:mtxvwtwo} are 
  formulated for a single parameter interpolation point $\mur \in \M$.
  However, the results directly extend to interpolation at multiple parameter
  sampling points $\mur^{(1)}, \ldots, \mur^{(q)} \in \M$  by constructing the
  projection spaces for every parameter sample  and then concatenating the
  resulting spaces into a single global projection space.
  As example, consider the task of interpolating
  \begin{align} \label{eqn:exampleparam}
    \begin{aligned}
      G_{1}(\sigma_{1}, \mur^{(1)}), &&
        G_{2}(\sigma_{1}, \sigma_{2}, \mur^{(1)}),\\
      G_{1}(\sigma_{3}, \mur^{(2)}), &&
        G_{2}(\sigma_{3}, , \sigma_{4}, \mur^{(2)}),
    \end{aligned}
  \end{align}
  with the four frequency points $\sigma_{1}, \sigma_{2}, \sigma_{3}, 
  \sigma_{4}$ and the two parameter points $\mur^{(1)}, \mur^{(2)}$.
  Using \Cref{thm:mtxvw} Part~(a), we can construct basis matrices $V^{(1)}$, 
  $V^{(2)}$ for the interpolation in either $\mur^{(1)}$ or $\mur^{(2)}$, 
  respectively.
  The construction of a reduced-order model that satisfies all interpolation 
  conditions~\cref{eqn:exampleparam} is then given by constructing $V$ such that
  \begin{align*}
    \mspan(V) \supseteq \mspan([V^{(1)},~V^{(2)}]).
  \end{align*}
\end{remark}

\begin{remark}
  The results simplify drastically for single-input single-output (SISO) systems.
  In that case, the multivariate transfer functions corresponding to bilinear 
  systems~\cref{eqn:tfmimo} can be written without Kronecker
  products~\cref{eqn:sisotf} and the construction of the corresponding projection
  spaces simplifies such that no Kronecker products are involved anymore.
\end{remark}


\section{Numerical examples}%
\label{sec:examples}

We illustrate the analysis with two benchmark examples.
The experiments reported here have been executed on a machine with 2 Intel(R)
Xeon(R) Silver 4110 CPU processors running at 2.10GHz and equipped with
192 GB total main memory.
The computer is run on CentOS Linux release 7.5.1804 (Core) with
MATLAB 9.7.0.1190202 (R2019b).


\subsection{Parametric bilinear time-delay system}%
\label{subsec:delayexample}

In the first example from~\cite{morGosPBetal19}, we consider a time-delayed
heated rod modeled by a one-dimensional heat equation
\begingroup
\small
\begin{align*}
  \partial_{t} v(\zeta, t)  = \partial_{\zeta}^2 v(\zeta, t)
    + a_{1}(\zeta) v(\zeta, t)
    + a_{2}(\zeta) v(\zeta, t - 1) + u(t),
\end{align*}
\endgroup
with homogeneous Dirichlet boundary conditions. We pa\-ram\-e\-ter\-ize 
the diffusivity using the coefficients
\begin{align*}
  a_{1} = -\mu \sin(\zeta)~\text{and}~a_{2} = \mu \sin(\zeta),
    ~~\text{for}~~\mu \in [1, 10].
\end{align*}
The non-parametric example in~\cite{morGosPBetal19} is recovered for $\mu = 2$.
After a  spatial discretization, we obtain a parametric bilinear system of the 
form
\begingroup
\small
\begin{align*}
  \begin{aligned}
      \dot{x}(t) & = (A_{0} - \mu A_{d}) x(t) + \mu A_{d} x(t - 1)
        + N x(t) u(t) + B u(t), \\
      y(t) & =  C x(t),
  \end{aligned}
\end{align*}
\endgroup
with $m = p = 1$ and $n = 5\,000$.
In our structured parametric setting, this model corresponds to the matrix 
functions
\begin{align*}
  \begin{aligned}
    \cK(s, \mu) & = sI_{n} - (A_{0} - \mu A_{d}) - \mu e^{-s} A_{d},\\
    \cB(s, \mu) & = B,~\cN(s, \mu) = N,~\text{and}~\cC(s, \mu) = C.\\
  \end{aligned}
\end{align*}
The reduced-order model is constructed via \Cref{thm:mtxvwtwo} 
Part (a) with  the frequency  sampling points $\{\pm 10^{-4}\mathrm{i},$
$\pm 10^{4}\mathrm{i}\}$ and the parameter sampling points $\{1, 5.5, 10\}$ for
the first two transfer functions.
By construction, the reduced-order model has the same parametric time-delay 
structure as the original model and the state-space dimension $r = 24$.

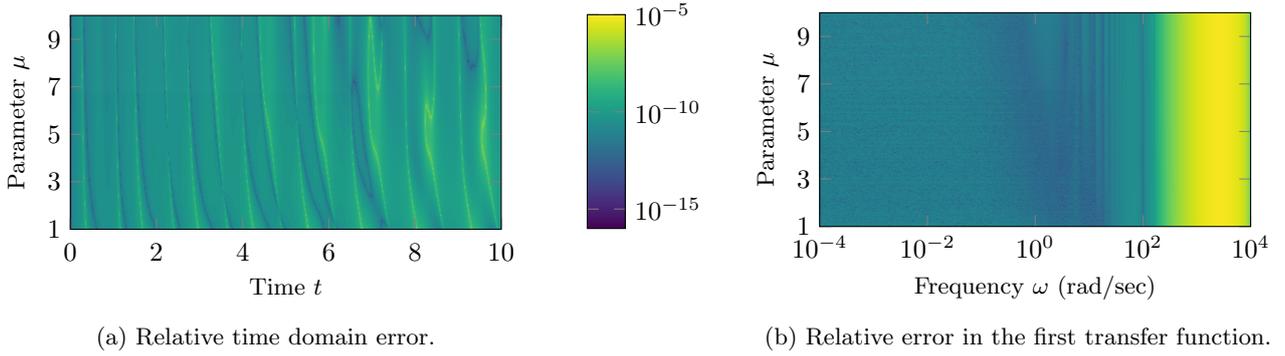
\begin{figure*}
\vspace{.5\baselineskip}
\begin{center}
    \begin{subfigure}[t]{.45\textwidth}
      \begin{center}
        \input{graphics/time_delay_sim.tikz}
        \subcaption{Relative time domain error.}
        \label{fig:time_delay_sim}
      \end{center}
    \end{subfigure}
    \hfill%
    \begin{minipage}[t]{.092\textwidth}
      \vspace{-9.85\baselineskip}
      \input{graphics/time_delay_legend.tikz}
    \end{minipage}%
    \hfill%
    \begin{subfigure}[t]{.45\textwidth}
      \begin{center}
        \input{graphics/time_delay_freq_g1.tikz}
        \subcaption{Relative error in the first transfer function.}
        \label{fig:time_delay_freq_g1}
      \end{center}
    \end{subfigure}
    \caption{Relative errors for the time-delay system.}
    \label{fig:time_delay}
  \end{center}
\end{figure*}

\Cref{fig:time_delay_sim} shows the relative time response
error in the output, given  by
\begin{align*}
  \err_{\td, \mtime}(t, \mu) := \frac{\lvert y(t; \mu) - \yr(t; \mu) \rvert}
    { \lvert y(t; \mu) \rvert},
\end{align*}
for $t \in [0, 10]$ and $\mu \in [1, 10]$, using the same test input signal as 
in~\cite{morGosPBetal19}, namely,
$u(t)  = 0.05 \left( \cos(10t) + \cos(5t) \right)$.
The maximum error in the time and parameter domain is 
\begin{align*}
  \max\limits_{\mu \in [1, 10]} \left( \max\limits_{t \in [0, 10]}
    \err_{\td, \mtime}(t, \mu) \right)
    \approx 9.993 \cdot 10^{-6},
\end{align*}
illustrating a high-fidelity parametric reduced model over the full parameter
domain.
\Cref{fig:time_delay_freq_g1} depicts the relative error in the first 
transfer function over the parameter range, computed as
\begin{align*}
  \err_{\td, \mfreq}(\omega_{1}, \mu) := \frac{\lvert
    G_{1}(\omega_{1}\mathrm{i}, \mu) -
    \Gr_{1}(\omega_{1}\mathrm{i}, \mu) \rvert}{\lvert 
    G_{1}(\omega_{1}\mathrm{i}, \mu)\rvert},
\end{align*}
where $\omega_{1} \in [10^{-4}, 10^{4}]$ and $\mu \in [1, 10]$.
As for the time domain error, we computed the maximum error to obtain
\begin{align*}
  \max\limits_{\mu \in [1, 10]} \left(
    \max\limits_{\omega_{1} \in [10^{-4}, 10^{4}]}
    \err_{\td, \mfreq}(\omega_{1}, \mu)
    \right)
    \approx 7.002 \cdot 10^{-6},
\end{align*}
showing the accuracy of the parametric reduced model in the frequency domain as
well.
We computed the maximum relative error in the second transfer function 
$G_{2}(s_{1}, s_{2}, \mu)$ as well to obtain
\begingroup
\small
\begin{align*}
  \max\limits_{\mu \in [1, 10]} \left(
    \max\limits_{\omega_{1}, \omega_{2} \in [10^{-4}, 10^{+4}]}
    \err_{\td, \mfreq}(\omega_{1}, \omega_{2}, \mu)
    \right) \approx 6.657 \cdot 10^{-4},
\end{align*}
\endgroup
where
\begingroup
\small
\begin{align*}
  \err_{\td, \mfreq}(\omega_{1}, \omega_{2}, \mu) :=
  \frac{\lvert G_{2}(\omega_{1}\mathrm{i}, \omega_{2}\mathrm{i}, \mu) - 
    \Gr_{2}(\omega_{1}\mathrm{i}, \omega_{2}\mathrm{i}, \mu) 
    \rvert}{\lvert G_{2}(\omega_{1}\mathrm{i}, \omega_{2}\mathrm{i}, \mu) 
    \rvert}.
\end{align*}
\endgroup
All these results show that the  structure-pre\-ser\-ving parametric 
reduced-order model is an accurate approximation of the original 
system over the full parameter domain.


\subsection{Parametric bilinear mechanical system}

As second example, we consider a parametrized version of the multi-input 
multi-output damped mass-spring system from~\cite{morBenGW20}, a special case of 
the model~\cref{eqn:bsosys}, given by
\begin{align*} 
  \begin{aligned}
    & M \ddot{x}(t;\mu) + D \dot{x}(t;\mu) + K x(t;\mu) =
      B_{\mathrm{u}}u(t) \\
    & \quad{}+{} \mu_{1} N_{\mathrm{p}, 1} x(t) u_{1}(t) + \mu_{2} 
      N_{\mathrm{p}, 2} x(t) u_{2}(t),\\
    & y(t;\mu) = C_{\mathrm{p}} x(t;\mu) 
      \dot{x}(t;\mu),
  \end{aligned}
\end{align*}
where $\mu = (\mu_1,\mu_2)$ is the parameter entering through the bilinear terms 
and all the other matrices are exactly as in~\cite{morBenGW20}, except for 
$C_{\mathrm{p}} $, which,  we set as $C_{\mathrm{p}} = [e_{2}, 
e_{n-3}]^{\trans}$, where $e_{j}$ denotes the $j$-th column of the 
$n$-dimensional identity matrix.
We have then $n = 1\,000$ masses, $m = 2$ inputs and $p = 2$ outputs.
The parameter set is $\M = [0,1] \times [0,1]$.  
Note that for $\mu = (0, 0)$, the system becomes linear.
In our setting, this parametric bilinear model corresponds to
\begin{align*}
  \begin{aligned}
    \cK(s, \mu) & = s^{2}M + s D + K,~\cB(s, \mu) = B_{\mathrm{u}},\\
    \cN(s, \mu) & = \begin{bmatrix} \mu_{1} N_{\mathrm{p}, 1} &
      \mu_{2} N_{\mathrm{p}, 2}\end{bmatrix},~\text{and}~ \cC(s, \mu) = C_{\mathrm{p}}.\\
  \end{aligned}
\end{align*}
The reduced-order model is constructed via \Cref{thm:mtxvw}
with frequency interpolation points $\{\pm 10^{-4}\mathrm{i},$
$\pm 10^{4}\mathrm{i}\}$ and the parameter interpolation points
$\{ (0, 1), (1, 0) \}$ for the first two transfer functions.
To preserve the structural properties, such as positive definiteness of the mass 
matrix, of the single matrices, we  use a one-sided projection, i.e., we choose 
$W = V$.
Since the first transfer function (the linear term) is independent of the 
parameter, some of the vectors in the construction of $V$ are redundant and
removed, yielding a structured parametric reduced-order model with $r = 40$.
We compute similar error quantities as in \Cref{subsec:delayexample}.

\Cref{fig:mimo_msd_sim} illustrates the relative time domain output error over 
the parameter range $\mu \in [0, 1]^{2}$, computed as
\begin{align*}
  \err_{\msd, \mtime}(\mu) := \max\limits_{j \in  \{ 1, 2 \}} \left(
    \max\limits_{t \in [0, 100]} \frac{\lvert y_{j}(.; \mu) - \yr_{j}(.; \mu)
    \rvert}{\lvert y_{j}(.; \mu) \rvert} \right),
\end{align*}
using the input signal $u(t)  = \begin{bmatrix} \sin(200 t) + 200 \\
-\cos(200 t) - 200 \end{bmatrix}$.
The maximum error over the full parameter range is
\begin{align*}
  \max\limits_{\mu \in [0, 1]^{2}} \err_{\msd, \mtime}(\mu)
    \approx 8.849 \cdot 10^{-5},
\end{align*}
illustrating the high accuracy of the reduced model.
\Cref{fig:mimo_msd_freq_g1} shows the relative error in the first transfer 
function approximation, i.e.,
\begin{align*}
  \err_{\msd, \mfreq}(\omega_{1}) :=
    \frac{\lVert G_{1}(\omega_{1} \mathrm{i}) - \Gr_{1}(\omega_{1} \mathrm{i}) 
    \rVert_{2}}{\lVert G(\omega \mathrm{i}) \rVert_{2}},
\end{align*}
over the frequency range $\omega_{1} \in [10^{-4}, 10^{+4}]$, with the maximum 
attained error 
\begin{align*}
  \max\limits_{\omega_{1} \in [10^{-4}, 10^{+4}]}
    \err_{\msd, \mfreq}(\omega_{1})
    \approx 1.296 \cdot 10^{-4}.
\end{align*}
This error term is independent of the parameter since the first transfer 
function does not contain the parametric bilinear terms.
We also computed the maximum relative approximation error for the second 
transfer function as
\begin{align*}
  \max\limits_{\mu \in [0, 1]^{2}} \left(
    \max\limits_{\omega_{1}, \omega_{2} \in [10^{-4}, 10^{+4}]}
    \err_{\msd, \mfreq}(\omega_{1}, \omega_{1}, \mu)
    \right) \approx 1.496 \cdot 10^{-3},
\end{align*}
where
\begin{align*}
  \err_{\msd, \mfreq}(\omega_{1}, \omega_{1}, \mu) := 
    \frac{\lVert G_{2}(\omega_{1}\mathrm{i}, \omega_{2}\mathrm{i}, \mu) - 
    \Gr_{2}(\omega_{1}\mathrm{i}, \omega_{2}\mathrm{i}, \mu) 
    \rVert_{2}}{\lVert G_{2}(\omega_{1}\mathrm{i}, \omega_{2}\mathrm{i}, \mu) 
    \rVert_{2}}.
\end{align*}
These numbers illustrate that the structured parametric approximation is a  
high-fidelity surrogate both in the frequency and time domains.


\section{Conclusions}%
\label{sec:conclusions}

We have presented a structure-preserving interpolation framework for model order
reduction of parametric bilinear systems.
We have established the subspace conditions to enforce interpolation
both in the frequency and parameter domains.
Two numerical examples illustrate that the approach is well suited for 
efficient structure-preserving model order reduction of parametric bilinear 
systems.
The presented approach covers arbitrary parameter dependencies of the 
system as well as more system structures than shown in the examples.
An important open question is the appropriate choice of interpolation 
points in the frequency as well as the parameter domains to
minimize the approximation error in some appropriate measure.

\begin{figure}[t]
\begin{center}
    \begin{subfigure}[t]{.48\textwidth}
      \begin{center}
        \vspace{0em}
        
        \input{graphics/mimo_msd_sim.tikz}
        \vspace{-\baselineskip}
        
        \subcaption{Relative time domain error.}
        \label{fig:mimo_msd_sim}
      \end{center}
    \end{subfigure}
    \begin{subfigure}[t]{.48\textwidth}
      \begin{center}
        \vspace{.5\baselineskip}
      
        \input{graphics/mimo_msd_freq_g1.tikz}
        
        \subcaption{Relative error in the first transfer function.}
        \label{fig:mimo_msd_freq_g1}
      \end{center}
    \end{subfigure}
    
    \caption{Relative errors for the damped mass-spring system.}
    \label{fig:mimo_msd}
  \end{center}
\end{figure}
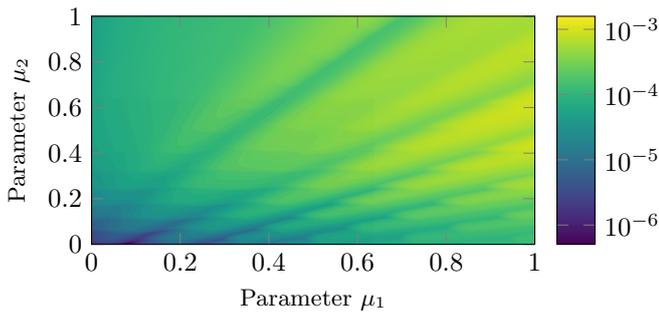
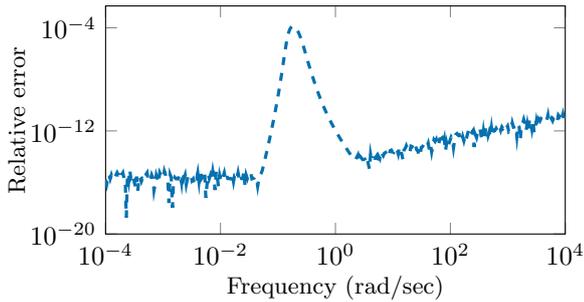


\section*{Acknowledgment}%
\addcontentsline{toc}{section}{Acknowledgment}

Benner and Werner were supported by the German Research Foundation
(DFG) Research Training Group 2297 
\textquotedblleft{}MathCoRe\textquotedblright{}, Magdeburg.
Gugercin was supported in parts by National Science Foundation under Grant No. 
DMS-1720257 and DMS-1819110.
Part of this material is  based upon work supported by the National Science 
Foundation under Grant No. DMS-1439786 and by the Simons Foundation Grant No. 
507536  while Gugercin and Benner were in residence at the 
Institute for Computational and Experimental Research in Mathematics in 
Providence, RI, during the \textquotedblleft{}Model and dimension reduction in 
uncertain and dynamic systems\textquotedblright{} program.


\addcontentsline{toc}{section}{References}
\bibliographystyle{plainurl}
\bibliography{bibtex/myref}

\end{document}

%% file: graphics/time_delay_sim.tikz
\begin{tikzpicture}
  \begin{axis}[
    view   = {0}{90},
    width  = .7\textwidth,
    height = .35\textwidth,
    scale only axis,
    axis on top,
    xmin   = 0,
    xmax   = 10,
    ymin   = 1,
    ymax   = 10,
    x tick label style = {yshift = -.175em},
    ytick  = {1, 3, 5, 7, 9},
    xlabel = {\small Time $t$\vphantom{Pp}},
    ylabel = {\small Parameter $\mu$}]
        
      \addplot graphics[xmin = 0, xmax = 10, ymin = 1, ymax = 10]
        {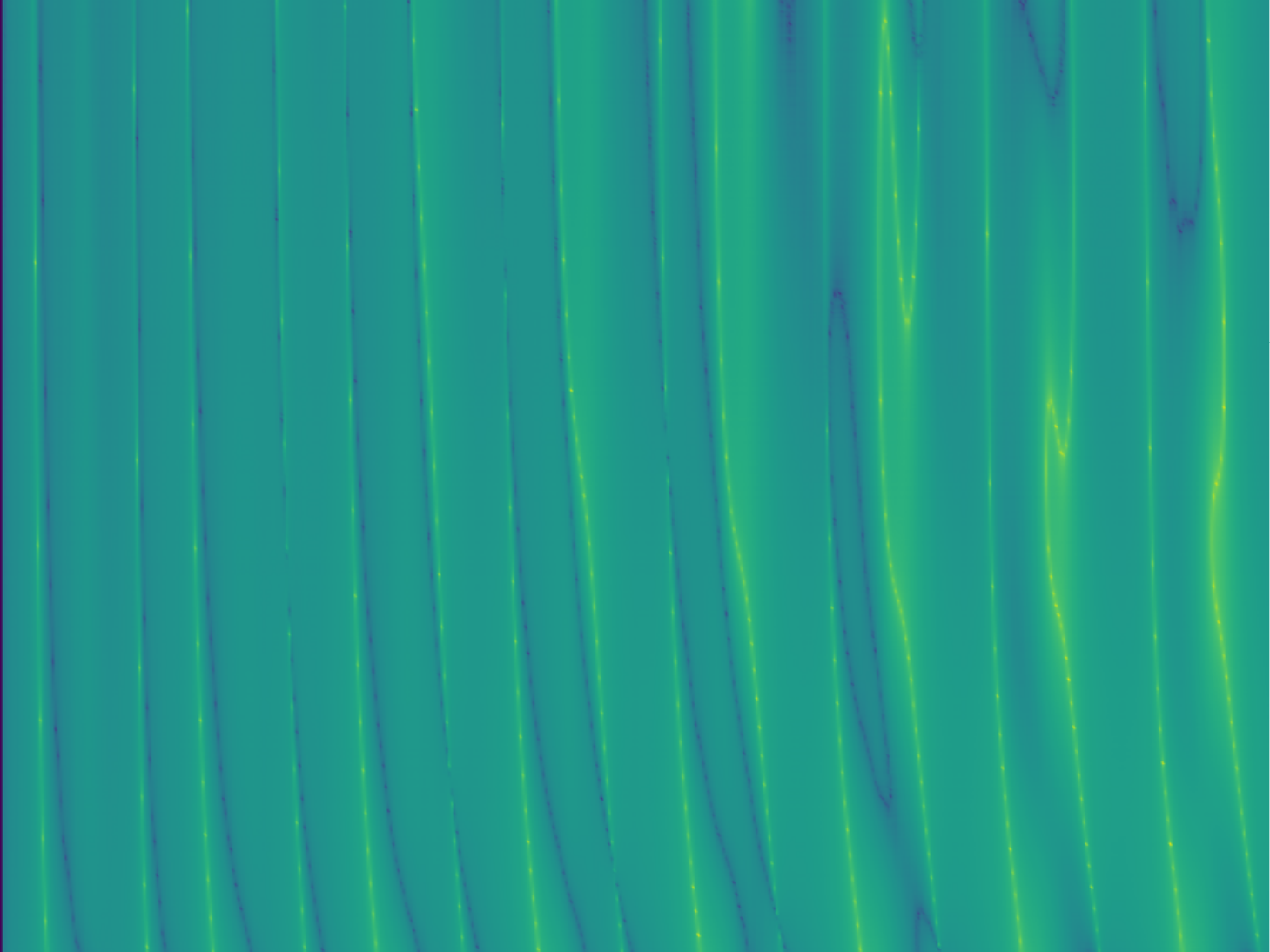};
            
  \end{axis}
\end{tikzpicture}

%% file: graphics/time_delay_legend.tikz
\begin{tikzpicture} 
  \begin{axis}[%
    hide axis,
    scale only axis,
    width  = .1cm,
    height = 1.712\textwidth,
    xmin = 0,
    xmax = 1,
    ymin = 0,
    ymax = 1,
    point meta min = -16,
    point meta max = -5,
    colorbar,
    colorbar style = {
      yticklabel = $10^{\pgfmathparse{\tick}
        \pgfmathprintnumber\pgfmathresult}$,
      anchor = north}]
  \end{axis}
\end{tikzpicture}

%% file: graphics/time_delay_freq_g1.tikz
\begin{tikzpicture}
  \begin{semilogxaxis}[
    view   = {0}{90},
    width  = .7\textwidth,
    height = .35\textwidth,
    scale only axis,
    axis on top,
    xmin   = 1e-4,
    xmax   = 1e+4,
    ymin   = 1,
    ymax   = 10,
    xtick  = {1e-4, 1e-2, 1e+0, 1e+2, 1e+4},
    ytick  = {1, 3, 5, 7, 9},
    xlabel = {\small Frequency $\omega$ (rad/sec)},
    ylabel = {\small Parameter $\mu$}]
        
      \addplot graphics[xmin = 1e-4, xmax = 1e+4, ymin = 1, ymax = 10]
        {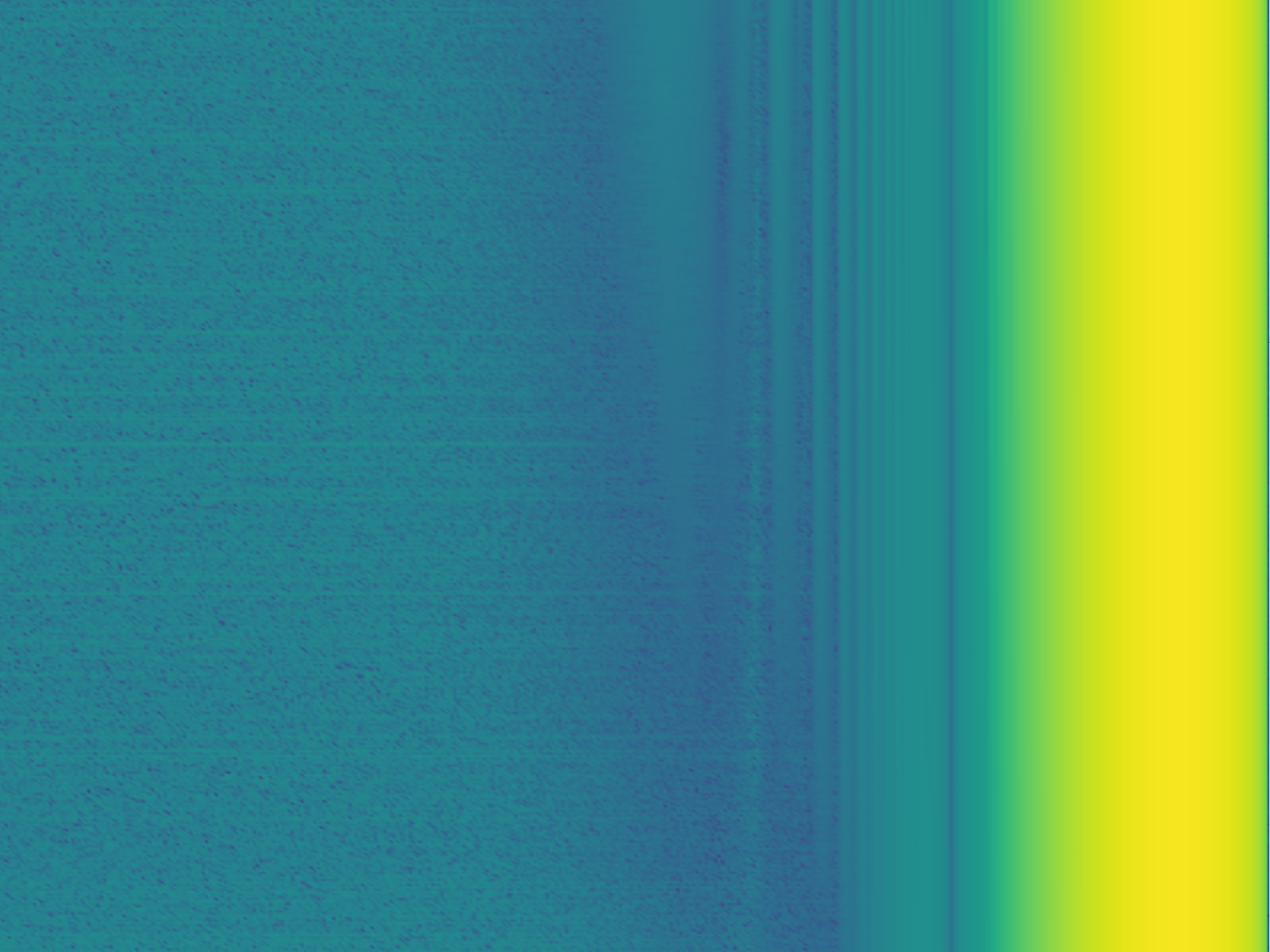};
            
  \end{semilogxaxis}
\end{tikzpicture}

%% file: graphics/mimo_msd_sim.tikz
\begin{tikzpicture}
%
%
  \begin{axis}[
    view   = {0}{90},
    width  = .675\textwidth,
    height = .35\textwidth,
    scale only axis,
    axis on top,
    xmin   = 0,
    xmax   = 1,
    ymin   = 0,
    ymax   = 1,
    xlabel = {\small Parameter $\mu_{1}$},
    ylabel = {\small Parameter $\mu_{2}$},
    point meta min = -6.2965,
    point meta max = -2.8000,
    colorbar,
    colorbar style = {
      yticklabel = $10^{\pgfmathparse{\tick}
      \pgfmathprintnumber\pgfmathresult}$,
      at = {(1.05, .5)},
      anchor = west}
   ]
        
      \addplot graphics[xmin = 0, xmax = 1, ymin = 0, ymax = 1]
        {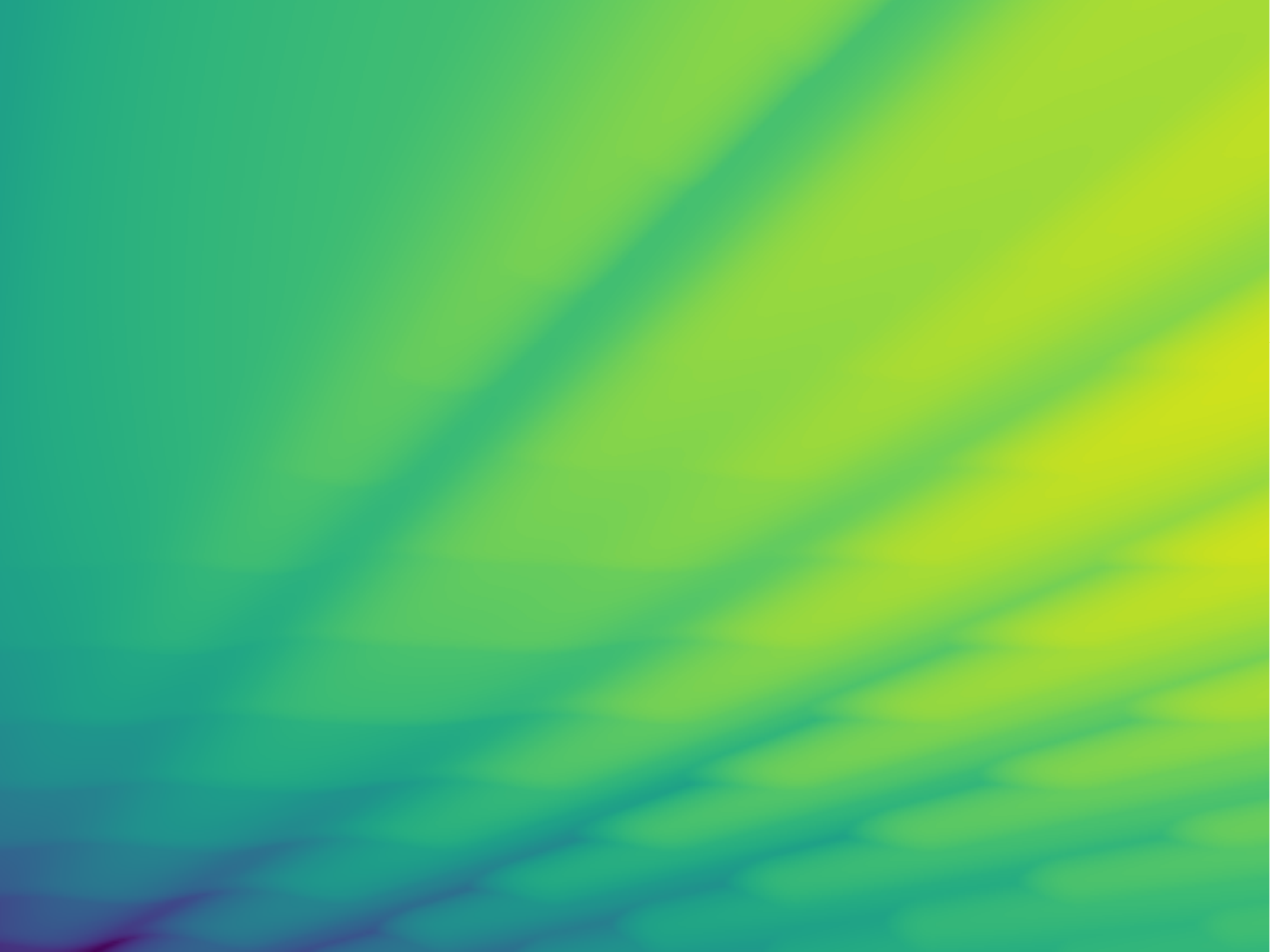};
            
  \end{axis}
\end{tikzpicture}

%% file: graphics/mimo_msd_freq_g1.tikz
\begin{tikzpicture}
  \pgfplotstableread{graphics/data/mimo_msd_freq_g1.dat}\tableRELERR
  
  \begin{loglogaxis}[%
    width  = .7\textwidth,
    height = .35\textwidth,
    scale only axis,
    xmin = 1e-4,
    xmax = 1e+4,
    xtick = {1e-4, 1e-2, 1e+0, 1e+2, 1e+4},
    ymin = 1e-20,
    ymax = 5e-3,
    xminorticks = false,
    yminorticks = false,
    xlabel={\small Frequency (rad/sec)},
    xlabel style = {yshift = .3em},
    ylabel = {\small Relative error},
    ylabel style = {yshift = -.5em},
    scaled x ticks = false,
    x tick label style = {/pgf/number format/fixed},
    legend style = {
      at     = {(.5, -.395)},
      anchor = north,
      minimum width = 1.2cm,
      align = center}]
    
    \addplot[dashed, myBlue, line width = 1.2pt]
      table[x index = 0, y index = 3] {\tableRELERR};
  \end{loglogaxis}
\end{tikzpicture}

%% file: paper.bbl
\begin{thebibliography}{10}

\bibitem{morAlbFB93}
S.~Al-Baiyat, A.~S. Farag, and M.~Bettayeb.
\newblock Transient approximation of a bilinear two-area interconnected power
  system.
\newblock {\em Electric Power Systems Research}, 26(1):11--19, 1993.
\newblock \href {http://dx.doi.org/10.1016/0378-7796(93)90064-L}
  {\path{doi:10.1016/0378-7796(93)90064-L}}.

\bibitem{morAntBG10}
A.~C. Antoulas, C.~A. Beattie, and S.~Gugercin.
\newblock Interpolatory model reduction of large-scale dynamical systems.
\newblock In Javad Mohammadpour and Karolos~M. Grigoriadis, editors, {\em
  Efficient Modeling and Control of Large-Scale Systems}, pages 3--58. Springer
  US, 2010.
\newblock \href {http://dx.doi.org/10.1007/978-1-4419-5757-3_1}
  {\path{doi:10.1007/978-1-4419-5757-3_1}}.

\bibitem{morAntBG20}
A.~C. Antoulas, C.~A. Beattie, and S.~Gugercin.
\newblock {\em Interpolatory Methods for Model Reduction}.
\newblock Computational Science \& Engineering. Society for Industrial and
  Applied Mathematics, Philadelphia, PA, 2020.
\newblock \href {http://dx.doi.org/10.1137/1.9781611976083}
  {\path{doi:10.1137/1.9781611976083}}.

\bibitem{morBaiS06}
Z.~Bai and D.~Skoogh.
\newblock A projection method for model reduction of bilinear dynamical
  systems.
\newblock {\em Linear Algebra Appl.}, 415(2--3):406--425, 2006.
\newblock \href {http://dx.doi.org/10.1016/j.laa.2005.04.032}
  {\path{doi:10.1016/j.laa.2005.04.032}}.

\bibitem{morBauBBetal11}
U.~Baur, C.~A. Beattie, P.~Benner, and S.~Gugercin.
\newblock Interpolatory projection methods for parameterized model reduction.
\newblock {\em {SIAM} J. Sci. Comput.}, 33(5):2489--2518, 2011.
\newblock \href {http://dx.doi.org/10.1137/090776925}
  {\path{doi:10.1137/090776925}}.

\bibitem{morBauBF14}
U.~Baur, P.~Benner, and L.~Feng.
\newblock Model order reduction for linear and nonlinear systems: A
  system-theoretic perspective.
\newblock {\em Arch. Comput. Methods Eng.}, 21(4):331--358, 2014.
\newblock \href {http://dx.doi.org/10.1007/s11831-014-9111-2}
  {\path{doi:10.1007/s11831-014-9111-2}}.

\bibitem{morBeaG09}
C.~A. Beattie and S.~Gugercin.
\newblock Interpolatory projection methods for structure-preserving model
  reduction.
\newblock {\em Syst. Control Lett.}, 58(3):225--232, 2009.
\newblock \href {http://dx.doi.org/10.1016/j.sysconle.2008.10.016}
  {\path{doi:10.1016/j.sysconle.2008.10.016}}.

\bibitem{morBenGW20}
P.~Benner, S.~Gugercin, and S.~W.~R. Werner.
\newblock Structure-preserving interpolation of bilinear control systems.
\newblock e-print 2005.00795, arXiv, 2020.
\newblock math.NA.
\newblock URL: \url{https://arxiv.org/abs/2005.00795}.

\bibitem{morBenKS13}
P.~Benner, P.~K{\"u}rschner, and J.~Saak.
\newblock An improved numerical method for balanced truncation for symmetric
  second order systems.
\newblock {\em Math. Comput. Model. Dyn. Syst.}, 19(6):593--615, 2013.
\newblock \href {http://dx.doi.org/10.1080/13873954.2013.794363}
  {\path{doi:10.1080/13873954.2013.794363}}.

\bibitem{morBreD10}
T.~Breiten and T.~Damm.
\newblock {K}rylov subspace methods for model order reduction of bilinear
  control systems.
\newblock {\em Syst. Control Lett.}, 59(8):443--450, 2010.
\newblock \href {http://dx.doi.org/10.1016/j.sysconle.2010.06.003}
  {\path{doi:10.1016/j.sysconle.2010.06.003}}.

\bibitem{morConI07}
M.~Condon and R.~Ivanov.
\newblock Krylov subspaces from bilinear representations of nonlinear systems.
\newblock {\em Compel-Int. J. Comp. Math. Electr. Electron. Eng.},
  26(2):399--406, 2007.
\newblock \href {http://dx.doi.org/10.1108/03321640710727755}
  {\path{doi:10.1108/03321640710727755}}.

\bibitem{morFenB07a}
L.~Feng and P.~Benner.
\newblock A note on projection techniques for model order reduction of bilinear
  systems.
\newblock In {\em AIP Conference Proceedings}, volume 936, pages 208--211,
  2007.
\newblock \href {http://dx.doi.org/10.1063/1.2790110}
  {\path{doi:10.1063/1.2790110}}.

\bibitem{morGosPBetal19}
I.~V. Gosea, I.~Pontes~Duff, P.~Benner, and A.~C. Antoulas.
\newblock Model order reduction of bilinear time-delay systems.
\newblock In {\em 18th European Control Conference (ECC)}, pages 2289--2294,
  2019.
\newblock \href {http://dx.doi.org/10.23919/ECC.2019.8796085}
  {\path{doi:10.23919/ECC.2019.8796085}}.

\bibitem{Moh73}
R.~R. Mohler.
\newblock {\em Bilinear Control Processes: With Applications to Engineering,
  Ecology and Medicine}, volume 106 of {\em Mathematics in Science and
  Engineering}.
\newblock Academic Press, New York, London, 1973.

\bibitem{Ou10}
Y.~Ou.
\newblock {\em Optimal Control of a Class of Nonlinear Parabolic {PDE} Systems
  Arising in Fusion Plasma Current Profile Dynamics}.
\newblock PhD thesis, Lehigh University, Bethlehem, Pennsylvania, USA, 2010.

\bibitem{morRodGB18}
A.~C. Rodriguez, S.~Gugercin, and J.~Boggaard.
\newblock Interpolatory model reduction of parameterized bilinear dynamical
  systems.
\newblock {\em Adv. Comput. Math.}, 44(6):1887--1916, 2018.
\newblock \href {http://dx.doi.org/10.1007/s10444-018-9611-y}
  {\path{doi:10.1007/s10444-018-9611-y}}.

\bibitem{Rug81}
W.~J. Rugh.
\newblock {\em Nonlinear System Theory: The Volterra/Wiener Approach}.
\newblock The Johns Hopkins University Press, Baltimore, 1981.

\bibitem{SapSH19}
J.~Saputra, R.~Saragih, and D.~Handayani.
\newblock Robust ${H}_{\infty}$ controller for bilinear system to minimize
  {HIV} concentration in blood plasma.
\newblock {\em J. Phys.: Conf. Ser.}, 1245:012055, 2019.
\newblock \href {http://dx.doi.org/10.1088/1742-6596/1245/1/012055}
  {\path{doi:10.1088/1742-6596/1245/1/012055}}.

\bibitem{morScaA17}
G.~Scarciotti and A.~Astolfi.
\newblock Nonlinear model reduction by moment matching.
\newblock {\em Foundations and Trends\textsuperscript{\textregistered} in
  Systems and Control}, 4(3--4):224--409, 2017.
\newblock \href {http://dx.doi.org/10.1561/2600000012}
  {\path{doi:10.1561/2600000012}}.

\end{thebibliography}
